\documentclass{amsart}

\usepackage{amsthm}
\usepackage{amssymb}

\usepackage{graphicx}
\usepackage{psfrag}

 \theoremstyle{plain}
 \newtheorem{theorem}{Theorem}[section]
 \newtheorem{cor}[theorem]{Corollary}
 \newtheorem{lemma}[theorem]{Lemma}

 \newcommand{\cal}{\mathcal}

\newcommand{\imp}{\rightarrow}

\newcommand{\cp}{{\cal P}}

\newcommand{\cs}{{\cal S}}
\newcommand{\co}{{\cal O}}
\newcommand{\cd}{{\cal D}}

\newcommand{\ci}{{\cal I}}
\newcommand{\cw}{{\cal W}}

\title{On Generalized Total Colourings of Planar Graphs}

\author[Ph. Cara]{Philippe Cara}%
\address{Department of Mathematics and Data Science\\
  Vrije Universiteit Brussel\\
  1050 Brussel\\
  Belgium}
\email{philippe.cara@vub.be}

\author[S. Dorfling]{Samantha Dorfling}
\address{1970,Belgium}
\email{dorflingsamantha@gmail.com}

\keywords{generalized colouring, additive hereditary graph property,
  total colouring, maximal planar graph, planar graph, forest}
\subjclass[2010]{05C10, 05C15, 05C70}

\begin{document}
\maketitle

\begin{abstract}
  In this paper we study generalised total colourings of graphs where
  the colour classes formed by vertices and edges, respectively,
  induce forests, while incident edges/vertices receive distinct
  colours. In \cite{bobr16} it was conjectured that for planar graphs,
  four colours suffice for this type of colouring. We confirm this
  conjecture for two infinite families of planar graphs.

\end{abstract}


\section{Introduction}
In this paper we study simple, loopless, undirected graphs. For all
undefined graph theoretical terms we refer the reader to
\cite{ChartLes}. For a graph $G$, we will denote its set of vertices
and its set of edges by $V(G)$ and $E(G)$, respectively.

The colourings we are interested in were introduced in the general
context of hereditary properties, in~\cite{bokemi2011}. Although we will
only be concerned with one particular case of the so-called  $({\cal
  P},{\cal Q})$-total colourings,  we briefly sketch the broader
context for completeness, following \cite{bobr97}.
Let $\ci$ denote the class of all graphs. A \emph{graph property} is
any isomorphism-closed subclass of $\ci$. Let $\cp$ be such a
subclass. If a graph $G$ is a member of  
$\cp$ we will say that $G$ \emph{has property $\cp$}. A property $\cp$
is called \emph{additive} if for each graph 
$G$, all of whose components have the property $\cal P$, it follows that $G$
lies in $\cp$ too.  A property $\cp$ is said to be \emph{hereditary}
if, whenever $G$ lies in $\cp$, and $H$ is a subgraph of $G$, then $H$
also lies in $\cp$.

Well-known additive and hereditary graph
  properties are for example (see \cite{bobr97}):

  \begin{itemize}
  \item[] $\co = \{G \in \ci:E(G) = \emptyset \}$,
  \item[] $\cs_k = \{G \in \ci: \mbox{the maximum degree of $G$ is at
      most $k$} \}$,
  \item[] $\co_k = \{G \in \ci: \mbox{each component of $G$ has order
      at most $k+1$}\}$,
  \item[]
    $\cw_k = \{G \in \ci: \mbox{the length of the longest path in $G$
      is at most $k$} \}$,
  \item[] $\cd_k = \{G \in \ci: \mbox{ every subgraph of $G$ has
      minimum degree at most $k$}\}$.
  \end{itemize}

Let ${\cal P}$ and ${\cal Q}$ denote additive hereditary graph
properties. A \emph{$({\cal P},{\cal Q})$-total colouring} of a simple
graph $G$ is a colouring of the vertices and edges of $G$ such that for
each colour $i$, the vertices coloured by $i$ induce a subgraph of $G$
with property ${\cal P}$, the edges coloured by $i$ induce a subgraph
of $G$ with property ${\cal Q}$, and incident vertices and edges are
coloured distinctly. The minimum number of colours needed such that a
graph $G$ has a $({\cal P},{\cal Q})$-total colouring is called the
\emph{(${\cal P},{\cal Q}$)-total chromatic number} of $G$ and is denoted by
$\chi''_{{\cal P},{\cal Q}}(G)$.

Early studies of total colourings of graphs occur in \cite{bechco67},
\cite{be69} and \cite{Vija71}. In \cite{bokemi2011}, the
authors then generalised the idea of total colourings of graphs and
describe such colourings in the context of additive hereditary graph
properties. In particular, they found upper and lower bounds for the
generalized total colourings of graphs with various additive
hereditary properties.  

In \cite{bobr16}, they study generalised total colourings for families of
planar graphs and give upper bounds for these chromatic numbers for
proper colourings of the vertices while the monochromatic edge sets
are allowed to be forests. Using the terminology and symbols
introduced above, the total chromatic number
$\chi''_{{\cal O},{\cal 
    D}_1}(G)$ is studied for $G$ planar. In particular they show that
if an even planar triangulation has a Hamilton cycle $H$ for which
there is no cycle among the edges inside $H$, then such a graph needs
at most four colours for a $({\cal O},{\cal D}_1)$-total colouring as
described above.

Furthermore, in \cite{bobr16}, they conjecture that for all planar
graphs $G$ one must have 
$\chi''_{{\cal D}_1,{\cal D}_1}(G) \leq 4$. 
In this paper we define two classes of maximal planar graphs and show
that for all graphs $G$ in these classes, 
$\chi''_{{\cal D}_1,{\cal D}_1}(G) = 
4$. Therefore, we can confirm that for these two infinite classes of planar
graphs the above-mentioned conjecture in \cite{bobr16} is true. We
remark that, like in most results from~\cite{bobr16}, our graphs are
hamiltonian.

\section{Notation and basic results}

A \emph{$\cd_1$ vertex colouring} of a graph $G$ is a colouring of the
vertices of $G$ such that the subgraphs of $G$ induced by these colour
classes are forests (i.e. lie in ${\cal D}_1$). 

A planar graph $G$ is called {\em maximal planar} if the addition of
any edge to $G$ results in a nonplanar graph.

The following lemma is a well-known characterisation of maximal planar
graphs, found in~\cite{ChartLes} for example:
\begin{lemma}\label{MPGChara}
  A planar graph $G$ with order $n\geq 3$ and size $m$ is maximal
  planar if and only if $m=3n-6$.
\end{lemma}

A maximal planar graph that contains a hamiltonian cycle will be
called \emph{hamiltonian maximal planar}.

We will construct a first class of maximal planar hamiltonian graphs.

Suppose that $n \geq 4$ is any integer.
Let $G = (V(G),E(G))$ be the planar graph we obtain as follows: \\
Let $C_n$ denote a cycle on $n$ vertices; labelled $v_1, \ldots, v_n$. \\
Add the following edges $v_iv_j$ to $C_n$ (on the inside of $C_n$): \\
where $i+j=n+2$ for all $2 \leq i,j \leq n$ and \\
where $i+j = n+ 1$ for all $2 \leq i,j \leq n$. \\
Finally, if $n > 4$, then also add the edges $v_1v_i$ for all
$3 \leq i < n-1$ on the outside of the cycle. The class of all such
graphs $G$ will be denoted by $\mathsf{MH}$. Our construction is illustrated in
Figure~\ref{fig1}. 

 \psfrag{v1}{$v_{1}$}\psfrag{v2}{$v_{2}$}\psfrag{v3}{$v_{3}$}
  \psfrag{v4}{$v_{4}$}\psfrag{v5}{$v_{5}$}\psfrag{v6}{$v_{6}$}
  \psfrag{v7}{$v_{7}$}\psfrag{v8}{$v_{8}$}\psfrag{v9}{$v_{9}$}
  \psfrag{v10}{$v_{10}$}\psfrag{v11}{$v_{11}$}\psfrag{v12}{$v_{12}$}

\begin{figure}[hbtp]
  \centering
  \includegraphics[width=.6\textwidth]{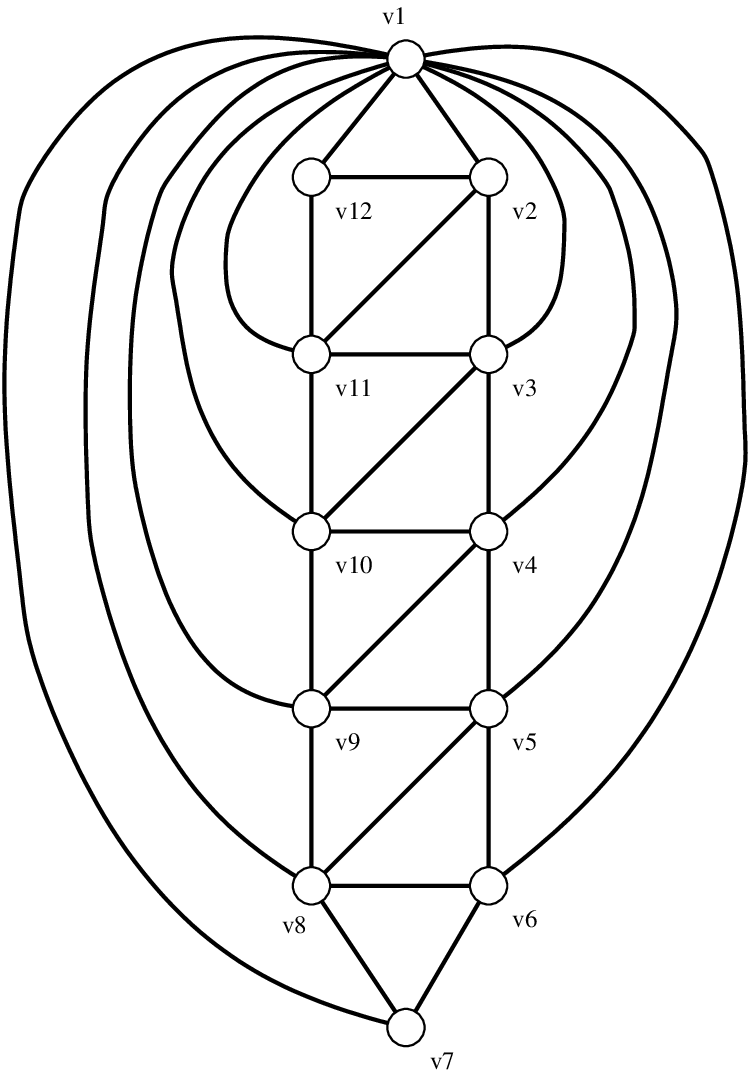}
  \caption{The member of $\mathsf{MH}$ based on $C_{12}$.}\label{fig1}
\end{figure}

\begin{lemma}\label{MaxHam}
  Every graph in the class $\mathsf{MH}$ is maximal planar and hamiltonian.
\end{lemma}
\begin{proof}
  Let $G$ be any graph in $\mathsf{MH}$ with order $n \geq 4$. Then clearly $G$
  is planar as well as Hamiltonian. From the construction above we
  obtain the size of $G$ in the case where $n$ is even and where $n$
  is odd. First suppose that $n$ is odd, then to find the size of $G$
  we add the $n$ cycle edges to the $\frac{n-3}{2}$ inside edges
  $v_iv_j$ (where $i+j = n+2$ and $2 \leq i \leq n$) and to this we
  add the $\frac{n-3}{2}$ inside edges $v_iv_j$ (where $i+j = n+1$ and
  $2 \leq i \leq n$) and then we add the $n-3$ outside edges $v_1v_i$
  (where $3 \leq i < n-1$). Therefore the size of $G$ is
  $n+ \frac{n-3}{2} + \frac{n-3}{2} + n-3 = 3n - 6$ and thus, by
  Lemma~\ref{MPGChara}, it follows that $G$ is maximal planar.

  Next suppose that $n$ is even, then adding the edges in the same
  order as above gives size
  $n+ \frac{n-2}{2} + \frac{n-4}{2} + n-3 = 3n - 6$ and so again, by
  Lemma~\ref{MPGChara}, $G$ is maximal planar.
\end{proof}

\section{Generalized total colouring in the class $\mathsf{MH}$}

The main result in this section will show that for $G\in \mathsf{MH}$, we have
$\chi''_{{\cal D}_1,{\cal D}_1}(G) = 4$. In order to prove this, we
will need a few lemmas.

First, it will be useful to define the following sets.  For all
$k \in \{0,1,2\}$ let
$S_k = \{ x \in \mathbb N \mid x \bmod 3 = k \}$.



\begin{lemma}\label{VertexCol}
  Let $G$ be any graph in $\mathsf{MH}$ with vertex set
  $V(G) = \{v_1,v_2, \ldots,v_n\}$. Then the vertex colouring
  $c: V(G) \imp \{1,2\}$ of $G$, defined as $c(v_1) = 1$ and
  for all $i \in S_2$ such that
  $2 \leq i \leq \lceil\frac{n}{2}\rceil$, $c(v_i) = c(v_{n-i}) = 1$
  and otherwise $c(v_i) = 2$, is a $\cd_1$ vertex colouring of $G$.
\end{lemma}
\begin{proof}
  Let $G$ be any graph in $\mathsf{MH}$ with vertex set
  $V(G) = \{v_1,v_2, \ldots,v_n\}$ and let $c$ be the colouring in the
  statement of the lemma. For $i=1,2$ we will use $\langle i\rangle$
  to denote the set of all vertices $v$ in $V(G)$ such that $c(v)=i$.

  By the construction of graphs in $\mathsf{MH}$ and by definition of $c$, we
  know that $v_1$ is a universal vertex in $G$ and we know that
  $c(v_1)=1$. Furthermore, for all vertices $v_i$ and $v_j$ in $V(G)$
  with $c(v_i) = c(v_j) = 1$ (i.e $v_i,v_j \in \langle 1\rangle$), if $i,j \neq 1$,
  then $v_iv_j \not\in E(G)$: since otherwise, if $v_iv_j$ is an edge
  in $G$, then, without loss of generality,
  $2 \leq i \leq \lceil\frac{n}{2}\rceil$ and $i \in S_2$ while
  $j = n - k$ for some $k \in S_2$ such that
  $2 \leq k \leq \lceil \frac{n}{2}\rceil$. However, then
  $i + j = (3l + 2) + n - (3m + 2) = 3(l - m) + n$ for some integers
  $l$ and $m$, and by the construction of graphs in $\mathsf{MH}$, the edge
  $v_iv_j$ cannot exist.  Therefore the subgraph of $G$ induced by the
 colour  class $\langle 1\rangle$ is acyclic.

  Furthermore, the subgraph of $G$ induced by 
  $\langle 2\rangle$ is a path\\
  $P_1: v_n,v_{n-1}, v_3,v_4,v_{n-3},v_{n-4},v_6,v_7, 
  v_{n-6},v_{n-7}, \ldots, v_{\lceil\frac{n}{2}\rceil}$ in the case
  where $n \bmod 6 \in \{0,1,5\}$ and a path
  $P_2: v_n,v_{n-1}, v_3,v_4, v_{n-3},v_{n-4},v_6,v_7,v_{n-6},
  v_{n-7}, \ldots, v_{\lceil\frac{n}{2}\rceil + 1}$ in the case where
  $n \bmod 6 \in \{2,3,4\}$. 
  Therefore the subgraph induced by the colour class
  $\langle 2\rangle$ is also acyclic and thus the result holds.
\end{proof}

\psfrag{u}[][]{$2$}\psfrag{v}[][]{$1$}

\begin{figure}[hbtp]
  \centering
    \includegraphics[height=.4\textheight]{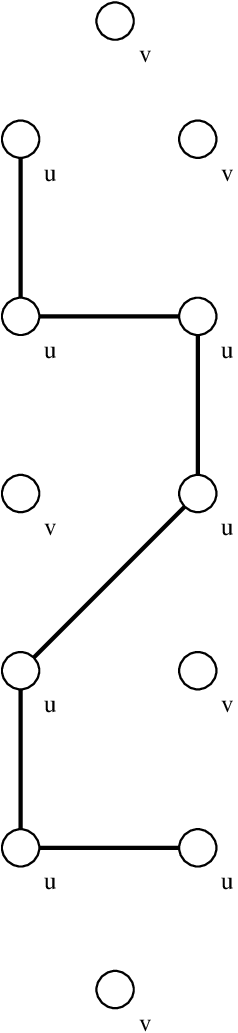}
    \qquad
    \includegraphics[height=.4\textheight]{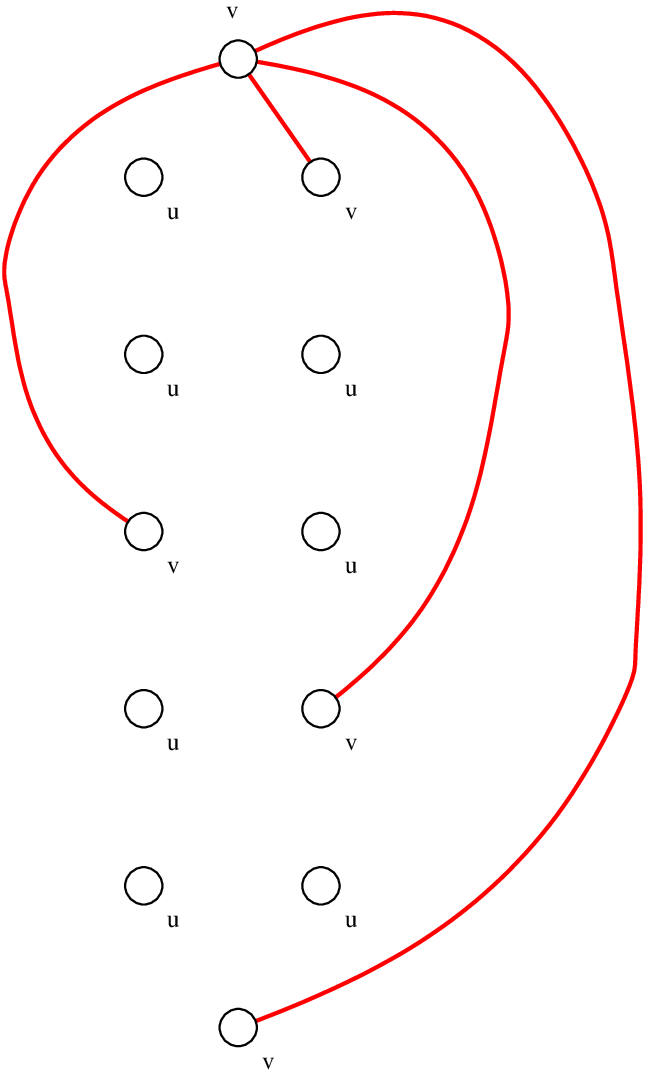}
  \caption{The vertex colouring and induced acyclic subgraphs in our
    example with $12$ vertices.}\label{fig2}
\end{figure}

For every graph $G$ in $\mathsf{MH}$ with order $n \geq 4$, let $G_n$ denote the graph obtained from $G$ as follows:
\begin{enumerate}
\item Colour the vertices of $G$ with the colouring $c$ defined in
  Lemma~\ref{VertexCol}.
\item Remove all edges from $G$ that join vertices with the
  same colour.
\end{enumerate}

The graph $G_{12}$ obtained from our example of order~$12$ is shown in
Figure~\ref{fig3}.

\begin{figure}[hbtp]
  \centering
  \includegraphics[height=.4\textheight]{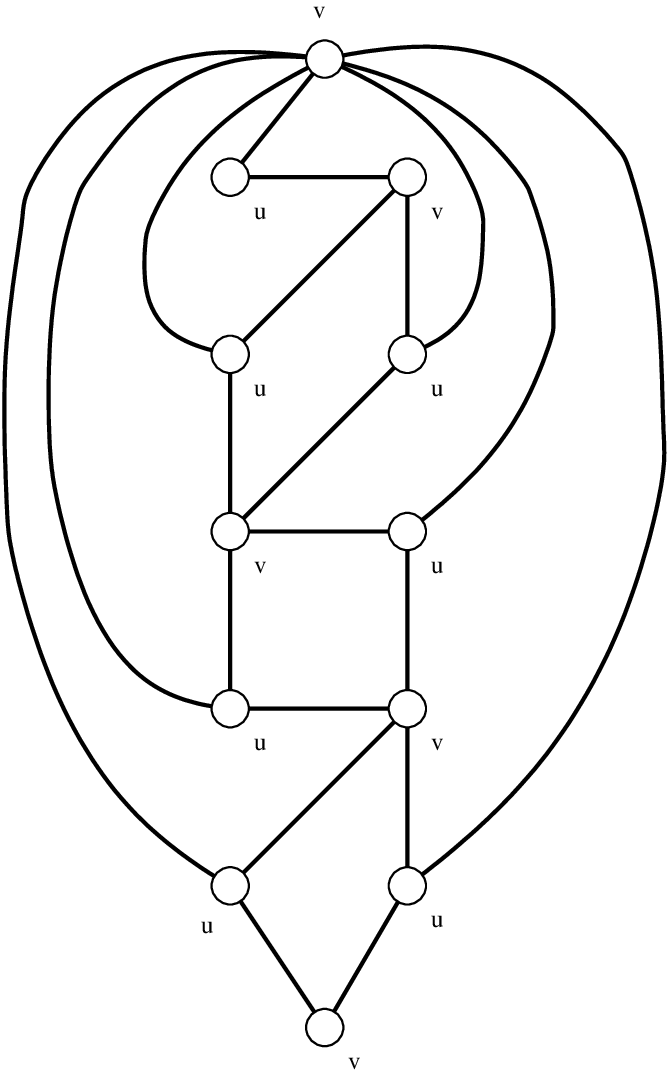}
  \caption{The subgraph $G_{12}$.}\label{fig3}
\end{figure}

Note that the graph $G_n$ is planar and that every edge of $G_n$ lies
on a cycle of length four. Furthermore, for all integers $n \geq 4$,
the graph $G_n$ contains a vertex of degree 2, and in particular, deg$_{G_n}(v_n) = 2$.

In the following lemma, let $G_n$ denote the graph constructed above,
with vertex set $V(G) = \{v_1,v_2, \ldots, v_n\}$, numbered as in the
construction above.

\begin{lemma}\label{NestedSG}
  For all positive integers $n \geq 5$, if $H$ and $G$ are graphs in
  $\mathsf{MH}$ with vertex sets $V(H) = \{u_1, u_2, \ldots ,u_{n-1}\}$ and
  $V(G) = \{v_1,v_2, \ldots,v_{n}\}$ respectively, then the graph
  $H_{n-1}$ is isomorphic to $G_n - v_{\lceil\frac{n+1}{2}\rceil}$ if
  $n$ is even, and isomorphic to $G_n - v_{\lceil\frac{n}{2}\rceil}$
  if $n$ is odd. In both cases the removed vertex has degree $2$ in
  $G_n$. 
\end{lemma}
\begin{proof}
  Take any integer $n \geq 5$ and let the graph $H_{n-1}$ be one that
  is constructed from a graph $H$ in $\mathsf{MH}$, with order $n-1$ and
  $V(H) = \{u_1, u_2, \ldots, u_{n-1}\}$. By Lemma~\ref{MaxHam}, $H$
  is a triangulation and we may choose an embedding in the plane such
  that the vertices 
  $u_1, u_{\lceil\frac{n-1}{2}\rceil}$ and
  $u_{\lceil\frac{n-1}{2}\rceil + 1}$ lie on the outer face of
  $H$. Note that any graph $G$ in $\mathsf{MH}$ (up to isomorphism) with order
  $n$, can be obtained 
  from $H$ by joining a single vertex $v$ with the vertices
  $u_1, u_{\lceil\frac{n-1}{2}\rceil}$ and
  $u_{\lceil\frac{n-1}{2}\rceil + 1}$ in $H$. One can then easily see
  the correspondence between 
  the vertices $V(G)=\{v_1,v_2,\ldots ,v_n\}$ used in the
  construction of $G$ and the vertices $\{u_1, u_2, \ldots,
  u_{\lceil\frac{n-1}{2}\rceil}, v, u_{\lceil\frac{n-1}{2}\rceil +
    1},\ldots , u_{n-1}\}=V(H)\cup\{v\}$. 

  Furthermore, if we colour the vertices of $G$ with the colouring $c$
  defined in Lemma~\ref{VertexCol}, then since this
  colouring is acyclic and
  $v_{\lceil\frac{n-1}{2}\rceil}v_{\lceil\frac{n-1}{2}\rceil+2}$ is an 
  edge in $G$ (corresponding to the edge
  $u_{\lceil\frac{n-1}{2}\rceil}u_{\lceil\frac{n-1}{2}\rceil+1}$ in
  $H$), that forms a cycle of length $3$ together with the vertex
  $v_1$ which has colour $1$, at least one of the vertices 
  $v_{\lceil\frac{n-1}{2}\rceil}$ and
  $v_{\lceil\frac{n-1}{2}\rceil+2}$ must have colour $2$. As the
  colouring is acyclic, the case
  $c(v_{\lceil\frac{n-1}{2}\rceil})=c(v_{\lceil\frac{n-1}{2}\rceil+2})=2$
  will yield $c(v_{\lceil\frac{n-1}{2}\rceil+1})=1$, implying that the
  degree of $v_{\lceil\frac{n-1}{2}\rceil+1}$ will be $2$ in $G_n$. In
  the case that $v_{\lceil\frac{n-1}{2}\rceil}$ and
  $v_{\lceil\frac{n-1}{2}\rceil+2}$ have different colours, the cycle
  consisting of $v_1$, $v_{\lceil\frac{n-1}{2}\rceil+1}$ and the
  vertex of colour $1$ in
  $\{v_{\lceil\frac{n-1}{2}\rceil},v_{\lceil\frac{n-1}{2}\rceil+2}\}$
  will yield colour $2$ for $v_{\lceil\frac{n-1}{2}\rceil+1}$, showing
  again that the degree of $v_{\lceil\frac{n-1}{2}\rceil+1}$ is $2$ in
  $G_n$. 

  Now, if $v_{\lceil\frac{n-1}{2}\rceil+1}$ is removed from the graph
  $G_n$, then the only edges 
  lost, come from the set
  $\{v_{\lceil\frac{n-1}{2}\rceil}v_{\lceil\frac{n-1}{2}\rceil+1},
  v_{\lceil\frac{n-1}{2}\rceil+1}v_{\lceil\frac{n-1}{2}\rceil+2},v_1v_{\lceil\frac{n-1}{2}\rceil+1}\}$
  and the remaining graph is isomorphic to $H_{n-1}$, as the colouring
  of $H$ is consistent with the colouring of $G$ constructed from $H$
  by adding the vertex $v$.

  To end the proof we remark that if $n$ is even we have
  $\lceil\frac{n-1}{2}\rceil+1=\lceil\frac{n}{2}\rceil+1$ and that if
  $n$ is odd we have
  $\lceil\frac{n-1}{2}\rceil+1=\lceil\frac{n}{2}\rceil$. So the vertex
  $v_{\lceil\frac{n-1}{2}\rceil+1}$ that we remove from $G_n$ is
  indeed equal to either $v_{\lceil\frac{n}{2}\rceil+1}$ or
  $v_{\lceil\frac{n}{2}\rceil}$, depending on the parity of $n$.
\end{proof}

\begin{lemma}\label{Decomp2Forests}
  Let $n \geq 4$ be any positive integer and $G$ be any graph in $\mathsf{MH}$
  with vertex set $V(G) = \{v_1,v_2, \ldots, v_n\}$. Then the edge set
  of the graph $G_n$ can be decomposed into two sets $F_1$ and $F_2$
  such that the graphs induced by each of these sets lie in
  ${\cal D}_1$.
\end{lemma}
\begin{proof}
  Let $G_n$ be the graph constructed from $G$ with vertex set
  $V(G_n) = V(G) = \{v_1,v_2, \ldots, v_n\}$. From
  Lemma~\ref{NestedSG}, there exists a vertex $v \in V(G_n)$ such that
  deg$(v) = 2$, in particular, if $n$ is odd, then one can take
  $v = v_{\lceil\frac{n}{2}\rceil}$ and if $n$ is even, then
  $v = v_{\lceil\frac{n+1}{2}\rceil}$ has degree~$2$.  To decompose the edges of
  $G_n$ as required, we consider two copies $F_1$ and $F_2$ of the
  empty graph $nK_1$, with the vertices of each labelled as in
  $G_n$. Consider the graph $G_n$. We now apply the following
  algorithm:

  Place one of the two edges incident with $v$ (in $G_n$), between the corresponding vertices in $F_1$ and the other edge between the corresponding vertices in $F_2$. \\
  Remove the vertex $v$ from $G_n$. Note that, by Lemma~\ref{NestedSG}, the resulting graph $G_n - v$ is isomorphic to a graph $G_{n-1}$ where $G_{n-1}$ contains a vertex $u$ such that deg$(u) = 2$. \\
  Continue with the above algorithm until the resulting graph is isomorphic to a cycle $G_4 \cong C_4$ of length four. \\
  Let $x$ denote any vertex in $G_4$. Place one of the two edges incident with $x$ (in $G_4$) between the corresponding vertices in $F_1$ and the other edge between the corresponding vertices in $F_2$. \\
  Remove $x$ from $G_4$. \\
  Finally place one of the two remaining edges in the resulting
  $G_2$ $(\cong P_3)$ between the corresponding vertices in $F_1$ and
  the other edge between the corresponding vertices in $F_2$.

  By the construction of $F_1$ and $F_2$ above it is easy to see that
  the graphs induced by both $F_1$ and $F_2$ are acyclic, since once
  the edge incident with a particular vertex $v$, in some graph $G_i$
  ($2 \leq i \leq n$), such that deg$(v) = 2$, is added to $F_j$,
  ($j \in \{1,2\}$), the vertex $v$ is never again encountered in
  $F_j$ as it is deleted from $G_i$.

  Using the above, we are able to decompose the edge set of $G_n$ into
  two sets such that the graph induced by edges in $F_1$ and the graph
  induced by edges in $F_2$ both lie in ${\cal D}_1$.
\end{proof}

\begin{figure}
  \centering
  \psfrag{u}{}\psfrag{v}{}
  
  \includegraphics[height=.3\textheight]{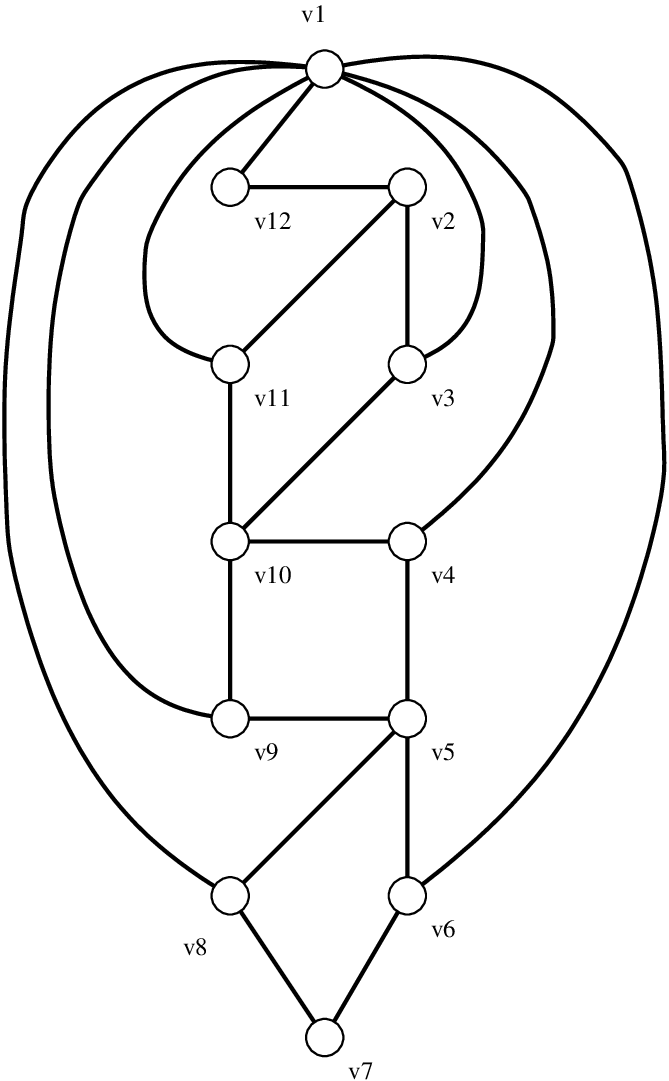}
  \qquad
  \includegraphics[height=.3\textheight]{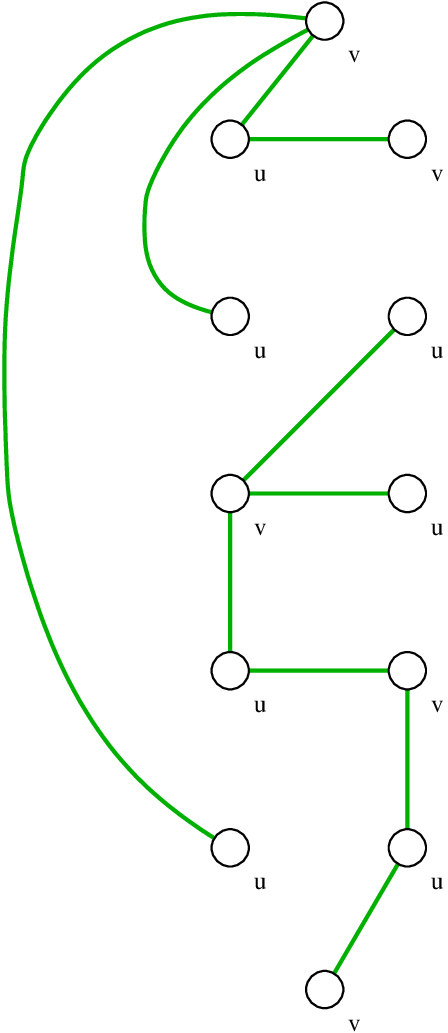}
  \qquad
  \includegraphics[height=.3\textheight]{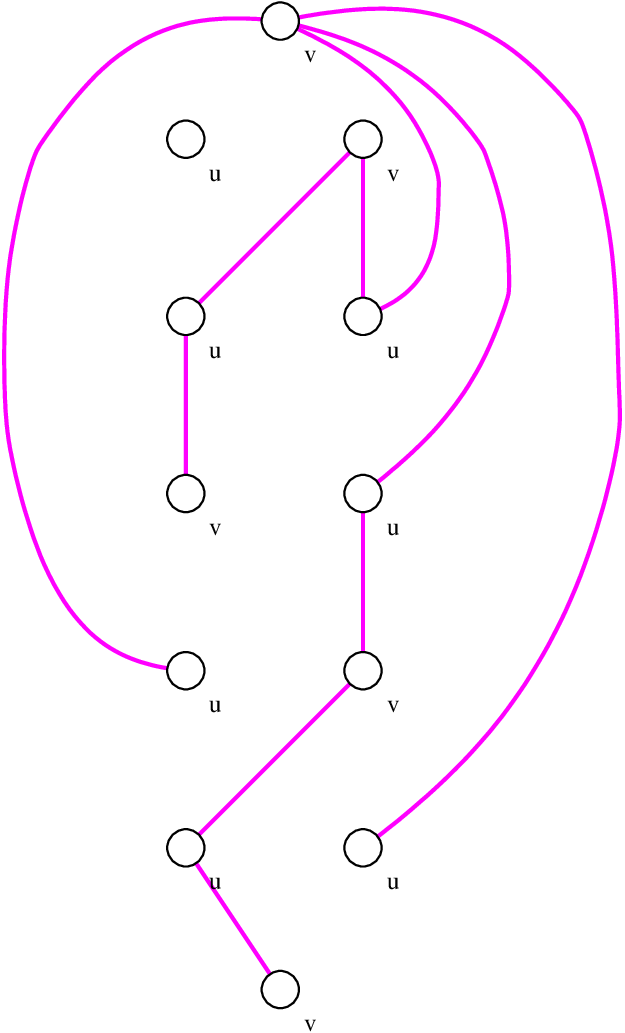}
  \caption{The decomposition of $G_{12}$ into two forests.}\label{fig4}
\end{figure}

\begin{theorem}\label{MainMH}
  For all integers $n \geq 4$, every graph $G$ in $\mathsf{MH}$ with order $n$
  satisfies $\chi''_{{\cal D}_1,{\cal D}_1}(G) = 4$.
\end{theorem}
\begin{proof}
  Let $G$ be any graph in $\mathsf{MH}$ with order $n \geq 4$ and vertex set
  $V(G) = \{v_1,v_2, \ldots, v_n\}$. Colour the vertices of $G$ with
  the colouring $c: V(G) \imp \{1,2\}$ of $V(G)$ such that
  $c(v_1) = 1$ and for all $i \in S_2$ such that
  $2 \leq i \leq \lceil\frac{n}{2}\rceil$, $c(v_i) = c(v_{n-i}) = 1$
  and otherwise $c(v_i) = 2$. By Lemma~\ref{VertexCol}, the subgraphs
  induced by the two colour classes $\langle 1\rangle$ and $\langle 2\rangle$ of $c$ are
  acyclic. Consider the subgraph $G_n$ of $G$. Applying the method used in the proof of
  Lemma~\ref{Decomp2Forests}, colour the edges of $G_n$ induced by
  $F_1$ with colour 3 and the edges induced by $F_2$ with 
 colour 4. Finally, for all edges $v_iv_j \in G$ such that
  $v_iv_j \not\in G_n$, ($1 \leq i,j \leq n$) if $c(v_i)=c(v_j) = 1$,
  then colour the edge $v_iv_j$ with colour 2, and for all edges
  $v_iv_j \in G$ such that $v_iv_j \not\in G_n$, ($1 \leq i,j \leq n$)
  if $c(v_i)=c(v_j) = 2$, colour the edge $v_iv_j$ with colour 1.

  The total colouring above provides a $\cd_1$ colouring of the
  vertices of $G$ as well as a $\cd_1$ colouring of the edges of $G$
  such that incident vertices and edges receive distinct colours and
  therefore \\ $\chi''_{{\cal D}_1,{\cal D}_1}(G) \leq 4$.

  Furthermore, the complete graph $K_4$ is the smallest graph in $\mathsf{MH}$
  and clearly $K_4$ is a subgraph of all
  graphs in $\mathsf{MH}$. However, if we try to find a total colouring of
  $K_4$, then we see that, firstly, we need at least two colours for
  the vertices and since $K_4$ is a triangulation, at most two
  vertices may have the same colour. Thus if we use exactly two
  colours for the vertices, then there are two vertices of each
  colour. However, then since $K_4$ is complete, we can always find a
  4-cycle $C: u_1,u_2,u_3,u_4,u_1$ in the graph with vertices of
  alternating colours. However, then none of the edges between the
  vertices on $C$ may receive colour 1 or 2. Furthermore, since $C$ is
  a cycle, we will need an additional two colours to colour its
  edges. However, if we use three colours for the vertices of $K_4$,
  then two vertices must have the same colour, say 1. Call
  the other two vertices $u$ and $v$ and call their colours $2$ and $3$
  respectively. Then 
  the only edge that may be coloured with $1$ is $uv$. Now the
  remaining five edges are the following: the edge between the two vertices
  of colour $1$, two vertices between $u$ and the vertices with colour $1$
  and two between $v$ and the vertices with colour $1$. They form two
  cycles of length $3$ sharing the edge between the vertices of colour
  $1$. Trying to use only $3$ colours, the edges on $u$ must be
  coloured with colour $3$ and the ones on $v$ with colour $2$. For
  the edge between the two vertices of colour $1$ either of the
  colours $2$ or $3$ will yield a monochromatic cycle of length $3$ so
  we will need a fourth colour. Therefore
  $\chi''_{{\cal D}_1,{\cal D}_1}(K_4) \geq 4$. Since $K_4$ is a
  subgraph of every graph in $\mathsf{MH}$, it follows that
  $\chi''_{{\cal D}_1,{\cal D}_1}(G) = 4$.
\end{proof}

The following corollary to Theorem~\ref{MainMH} follows from
\cite{bokemi2011}, since for a positive integer $k$, a subgraph $H$ of
a graph $G$ and additive hereditary properties ${\cal P}$ and
${\cal Q}$, if $\chi''_{{\cal P},{\cal Q}}(G) \leq k$, then
$\chi''_{{\cal P},{\cal Q}}(H) \leq k$.

\begin{cor}\label{MainPlanar}
  Every (planar) subgraph of a graph $G$ in $\mathsf{MH}$ satisfies \\
  $\chi''_{{\cal D}_1,{\cal D}_1}(G) \leq 4$.
\end{cor}

\section{Generalized total colourings for triangulated square grids} 

We now construct another family of hamiltonian maximal planar graphs.

Let $k \geq 2$ and $l \geq 2$ denote positive integers and $P_k$ and
$P_l$ denote paths of order $k$ and $l$, respectively. The cartesian
product of $P_k$ and $P_l$ is called a $k\times l$ \emph{grid} or a
$k\times l$ \emph{lattice graph} and has order $kl$.

Let $G$ be a $n\times n$ grid for some positive integer $n \geq 3$
and let $V(G)$ denote the vertex set of $G$. To label the vertices of
a square grid $G$, we will use notation similar to that used for entries of a
matrix. Thus $v_{ij}$ will denote the vertex in row $i$ and column $j$
of our grid. We will define a {\em triangulated grid} of $G$ to be the
graph $Gt$ obtained from $G$ by adding a diagonal edge from the top
right corner to the bottom left corner of every induced subgraph $H$
of $G$ such that $H$ is isomorphic to a cycle of length four, and we
will add an edge between the vertex $v_{11}$ and every vertex on the
outer face of $V(G)$.

More formally, if $V(G) = \{v_{11},v_{12}, \ldots,v_{nn}\}$, then $Gt$
is the graph obtained by adding all edges $v_{ij}v_{(i+1)(j-1)}$ (for
all $1 \leq i,j \leq n$) and all edges $v_{11}v_{1j}$, $v_{11}v_{i1}$,
$v_{11}v_{nj}$ and $v_{11}v_{in}$ for all $1 \leq i \leq n-1$ and
$2\leq j\leq n$ to
$G$. It is easy to see that the graph $Gt$ is planar. Furthermore, the
graph $Gt$ is maximal planar, since it has order $n^2$ and size
$2(n-1)n + (n-1)^2 + (4n - 4 - 3) = 3n^2 - 6$ (adding edges on the
grid $G$ to the diagonal edges and adding this to the order of the
neighbourhood of $v_{11}$ without the edges $v_{12}$ and  $v_{21}$).

 \psfrag{v31}{$v_{31}$}\psfrag{v32}{$v_{32}$}\psfrag{v33}{$v_{33}$}\psfrag{v34}{$v_{34}$}
  \psfrag{v41}{$v_{41}$}\psfrag{v42}{$v_{42}$}\psfrag{v43}{$v_{43}$}\psfrag{v44}{$v_{44}$}
  \psfrag{v21}{$v_{21}$}\psfrag{v22}{$v_{22}$}\psfrag{v23}{$v_{23}$}\psfrag{v24}{$v_{24}$}
  \psfrag{v13}{$v_{13}$}\psfrag{v11}{$v_{11}$}\psfrag{v12}{$v_{12}$}\psfrag{v14}{$v_{14}$}

\begin{figure}[hbtp]
  \centering
  \includegraphics[width=.6\textwidth]{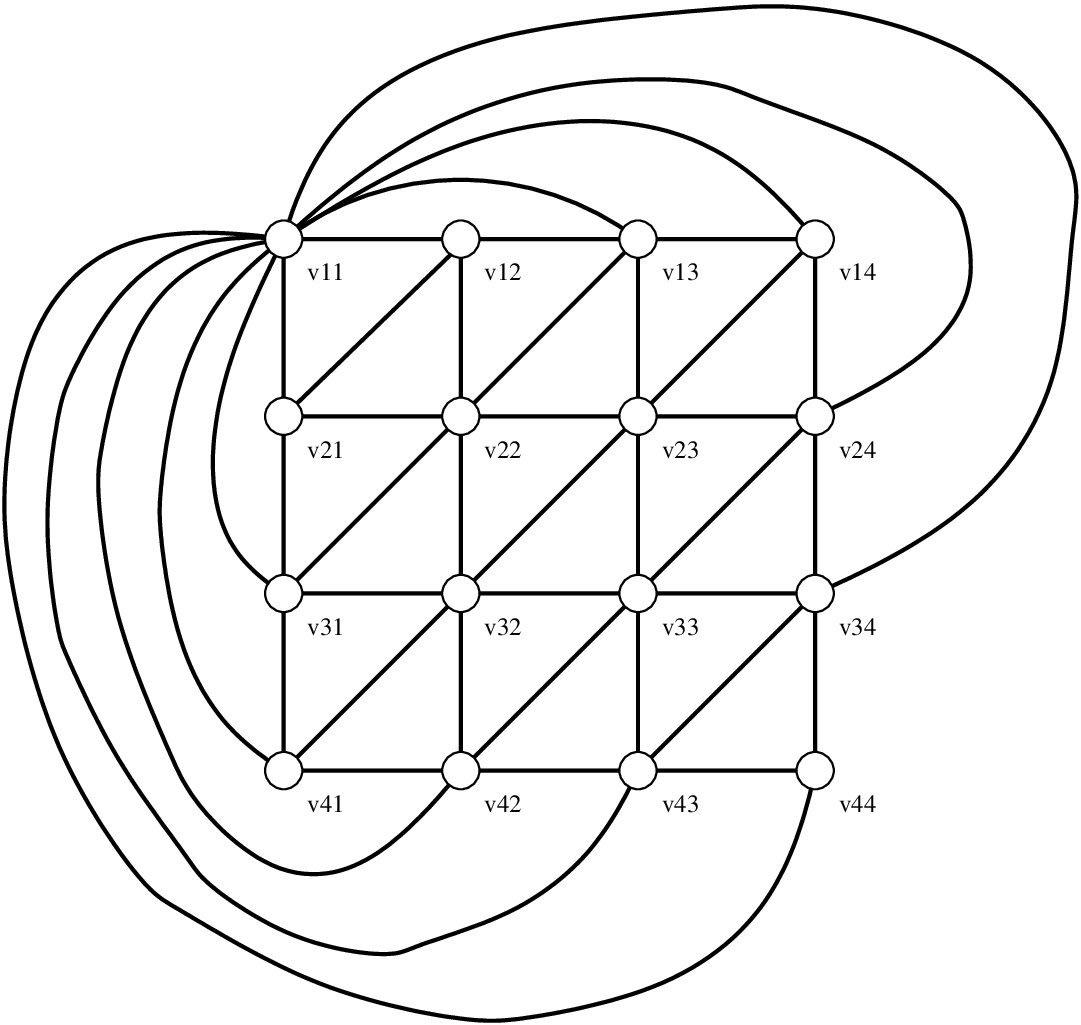}
  \caption{The triangulated $4\times 4$ grid.}\label{fig5}
\end{figure}

\begin{theorem}\label{MainTriGrid}
  Let $G$ be a $n\times n$ grid for any integer $n \geq 3$. Then
  the triangulated grid $Gt$ satisfies,
  $\chi''_{{\cal D}_1,{\cal D}_1}(Gt) = 4$.
\end{theorem}
\begin{proof}
  First we will show that $\chi''_{{\cal D}_1,{\cal D}_1}(Gt) \leq
  4$. Similar to the proof of Theorem~\ref{MainMH}, in order to show
  this upper bound, 
  we will first define a vertex colouring of a graph $Gt$ that
  will use two colours and we will show that the graphs induced by the
  two colour classes are both acyclic. Then we will show that the graph
  $H$ resulting from $Gt$ by removing all edges whose end vertices lie
  in the same colour class can be edge-partitioned into two spanning
  forests $F_1$ and $F_2$. From this we will be able to provide the
  necessary total colouring.

  Take any graph isomorphic to a triangulated grid $Gt$.

  Consider the vertex colouring $c: V(Gt) \imp \{1,2\}$ defined by

  $c(v_{ij}) = \left \{ \begin{array}{ll}
                          1 & \mbox{ if } 1 \leq i \leq j \mbox{ and } i+j \mbox{ is even, or} \\
                            & \mbox{ if } 1 \leq j \leq i-3, i \geq 4 \mbox{ and } i+j \mbox{ is odd, and } \\
                          2 & \mbox{ otherwise. }\end{array} \right.$

                      Let $\langle 1\rangle$ and $\langle 2\rangle$ denote the colour classes of
                      $c$, then the graph induced by the vertices in
                      $V(Gt)$ that lie in $\langle 1\rangle$, is acyclic and the
                      same holds for the graph induced by the vertices
                      that lie in $\langle 2\rangle$: this is easily seen, since by
                      definitions of the edge set $E(Gt)$ of $Gt$ and
                      the colouring $c$, the graph (say $G_2$) induced
                      on the vertices in $\langle 2\rangle$ is a
                      generalized caterpillar with, as spine, the path $v_{12},v_{21},v_{31}, \ldots,v_{(n-1)n}$ of length
                      $2n-2$ and legs of varying lengths $k$ such that
                      $0 \leq k \leq \lceil \frac{n}{2}
                      \rceil$. Furthermore, the graph (say $G_1$)
                      induced on the vertices of $\langle 1\rangle$ is
                      a generalized star with center $v_{11}$ and with
                      $2n-4$ arms of varying lengths $k$ such that 
                      $1 \leq k \leq \lceil \frac{n}{2} \rceil$.

                      Now let $H$ denote the spanning subgraph of $Gt$
                      that results when we remove the edge sets of
                      $G_1$ and $G_2$ from $Gt$.

                      Let $F_1$ denote the subgraph of $H$ induced by
                      the edges
                      $\{v_{ij}v_{kl} \in E(H) \ | \ i=k-1, j=l,
                      c(v_{ij}) = 1 \mbox{ and } c(v_{kl})=2\} \cup
                      \{v_{ij}v_{kl} \in E(H) \ | \ i=k, j=l-1,
                      c(v_{ij}) = 1 \mbox{ and } c(v_{kl})=2\} \cup
                      \{v_{ij}v_{kl} \in E(H) \ | \ i=k-1, j=l+1,
                      c(v_{ij}) = 1 \mbox{ and } c(v_{kl})=2\} \cup
                      \{v_{11}v_{kl} \in E(H) \ | \ k=1 \mbox{ and }
                      c(v_{kl}) = 2, \mbox{ or } l =n \mbox{ and }
                      c(v_{kl}) = 2\}$. It is easily seen that $F_1$
                      is a forest (a generalized star with $v_{11}$ its central vertex and with $n$ arms, such that every arm is a tree).
        
                      Let $F_2$ denote the subgraph of $H$ induced by
                      the edges
                      $\{v_{ij}v_{kl} \in E(H) \ | \ i=k-1, j=l,
                      c(v_{ij}) = 2 \mbox{ and } c(v_{kl})=1\} \cup
                      \{v_{ij}v_{kl} \in E(H) \ | \ i=k, j=l-1,
                      c(v_{ij}) = 2 \mbox{ and } c(v_{kl})=1\} \cup
                      \{v_{ij}v_{kl} \in E(H) \ | \ i=k-1, j=l+1,
                      c(v_{ij}) = 2 \mbox{ and } c(v_{kl})=1\} \cup
                      \{v_{11}v_{kl} \in E(H) \ | \ l=1 \mbox{ and }
                      c(v_{kl}) = 2, \mbox{ or } k =n \mbox{ and }
                      c(v_{kl}) = 2\}$. It is easily seen that $F_2$
                      is a forest (the union of a generalized star
                      with $n-1$ arms (such that the arms are paths)
                      together with a union of paths).

                      Therefore, to show that
                      $\chi''_{{\cal D}_1,{\cal D}_1}(Gt) \leq 4$, we
                      start with a graph $Gt$ and we colour its
                      vertices with the colouring $c$ defined above
                      to obtain two colour classes $\langle 1\rangle$ and $\langle 2\rangle$.  To
                       the edges of $Gt$, first we colour the
                      edges of the graph $G_1$ (defined above) with
                      the colour 2 and we colour the edges of the graph
                      $G_2$ with the colour 1. Next we colour the edges
                      of the forest $F_1$ with a new colour 3 and
                      finally the edges of the forest $F_2$ with a new
                     colour  4. From our observations above, we have
                      that the subgraphs induced by vertices with the
                      same colour lie in ${\cal D}_1$, all incident
                      vertices and edges have distinct colours and the
                      graphs induced by edges with the same colour lie
                      in ${\cal D}_1$.  Thus we may conclude that
                      $\chi''_{{\cal D}_1,{\cal D}_1}(Gt) \leq 4$.

Now we will show that $\chi''_{{\cal D}_1,{\cal D}_1}(Gt) \geq 4$. Note that the induced subgraph $\langle v_{11},v_{1(n-1)},v_{1n},v_{2(n-1)},v_{2n} \rangle$ of any triangulated grid $G_t$ is isomorphic to a wheel, which we will denote by $W$, which has order 5 and central vertex $v_{1n}$.   
We will show that $\chi''_{{\cal D}_1,{\cal D}_1}(W) \geq 4$, and therefore assume (to the contrary) that $\chi''_{{\cal D}_1,{\cal D}_1}(W) \leq 3$. Since $W$ contains cycles, we need at least two colours with which to colour the vertex set of $W$. For the sake of simplicity, let's label the vertices of $W$ as $\{u_1,u_2,u_3,u_4,x\}$, where $x$ denotes the central vertex of $W$.
We begin by colouring the vertices in $W$ that lie on the outer 4-cycle (which we will denote by $C$), i.e., all vertices excepting the vertex $x$. Consider the total colouring $c: \{u_1,u_2,u_3,u_4,x\} \cup E(W) \rightarrow \{1,2,3\}$ of $V(W) \cup E(W)$. \\
{\bf Case 1:} Suppose that $c(u_1) = c(u_3) \neq c(u_2)= c(u_4)$. Then, since incident vertices and edges must receive distinct colours and $C$ is a cycle, we will need two new colours to colour $E(C)$. Therefore $4 \leq \chi''_{{\cal D}_1,{\cal D}_1}(C) \leq \chi''_{{\cal D}_1,{\cal D}_1}(W) \leq 3$ --- a contradiction. \\
{\bf Case 2:} Suppose now that $c(u_1)= c(u_4) \neq c(u_2) = c(u_3)$ and, without loss of generality, $c(u_1) = 1$ and $c(u_2)=2$. Then, in $W$, the vertex $x$ must satisfy $c(x) = 3$.   
In order to satisfy the restrictions imposed by the total colouring,
we know that $c(u_1x) = c(u_4x) = 2$ and $c(u_2x) = c(u_3x) =
1$. However, then $c(u_1u_4)=3$ and thus the only available colour
assignment for $u_2u_3$ is $c(u_2u_3)=3$ --- now all the edges of $C$
are coloured the same, a contradiction. \\ 
{\bf Case 3:} Finally, suppose that $c(u_1)=c(u_2)= c(u_3) \neq c(u_4)$ and (without loss of generality) that $c(u_1) = 1$ and $c(u_4) = 2$. Clearly, $c(x) \neq 1$. \\
If $c(x) = 2$, then $c(u_1x) = c(u_3x) = c(u_3u_4) = c(u_1u_4) = 3$ and thus the colour class $\langle 3 \rangle$ induces a cycle $u_1,x,u_3,u_4$ --- a contradiction. 
Thus $c(x) = 3$. However, this forces $c(u_1u_4) = c(u_3u_4) = 3$ and $c(u_2x)=c(u_3x) = 2$ and the latter colouring forces $c(u_2u_3)=3$. This makes it impossible for us to colour the edge $u_1u_2$ with one of the three colours since $c(u_1u_2) \neq 1$ as $c(u_1) = c(u_2)=1$; also $c(u_1u_2) = 3$ causes $\langle 3 \rangle$ to induce a cycle isomorpic to $C$ and finally $c(u_1u_2) = 2$ results in a triangle $u_1,u_2,x$ induced by $\langle 2\rangle$. \\
Therefore $\chi''_{{\cal D}_1,{\cal D}_1}(W) \geq 4$. Since every
triangulated square grid $Gt$ contains the wheel $W$, we may conclude that $\chi''_{{\cal D}_1,{\cal D}_1}(Gt) \geq 4$ from which the result follows. 
                   \end{proof}
\begin{figure}[hbtp]
  \centering
    \includegraphics[height=.20\textheight]{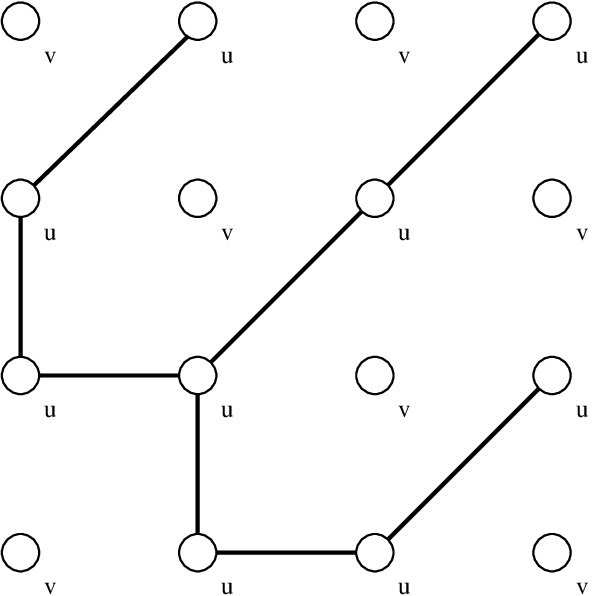}
    \qquad
    \includegraphics[height=.25\textheight]{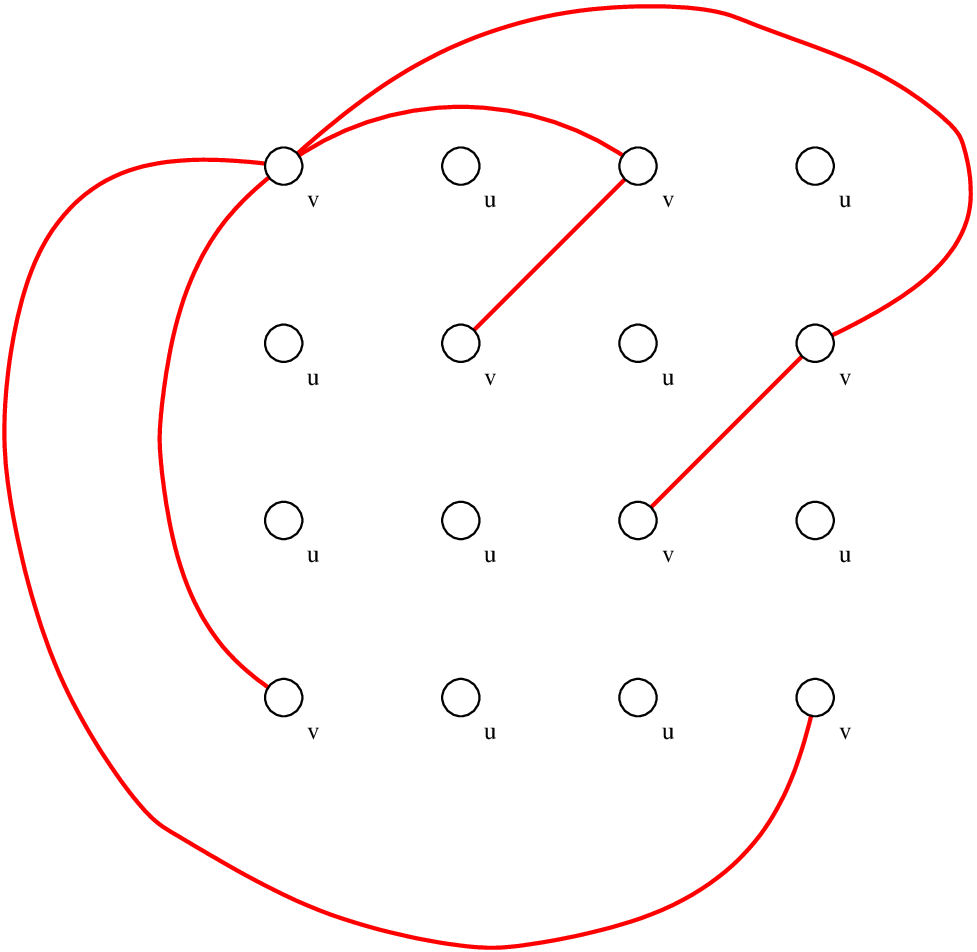}
  \caption{The vertex colouring and induced acyclic subgraphs in the
    triangulated $4\times 4$ grid.}\label{fig6}
\end{figure}

\begin{figure}[hbtp]
  \centering
  \psfrag{u}{}\psfrag{v}{}
  
  \includegraphics[height=.35\textheight]{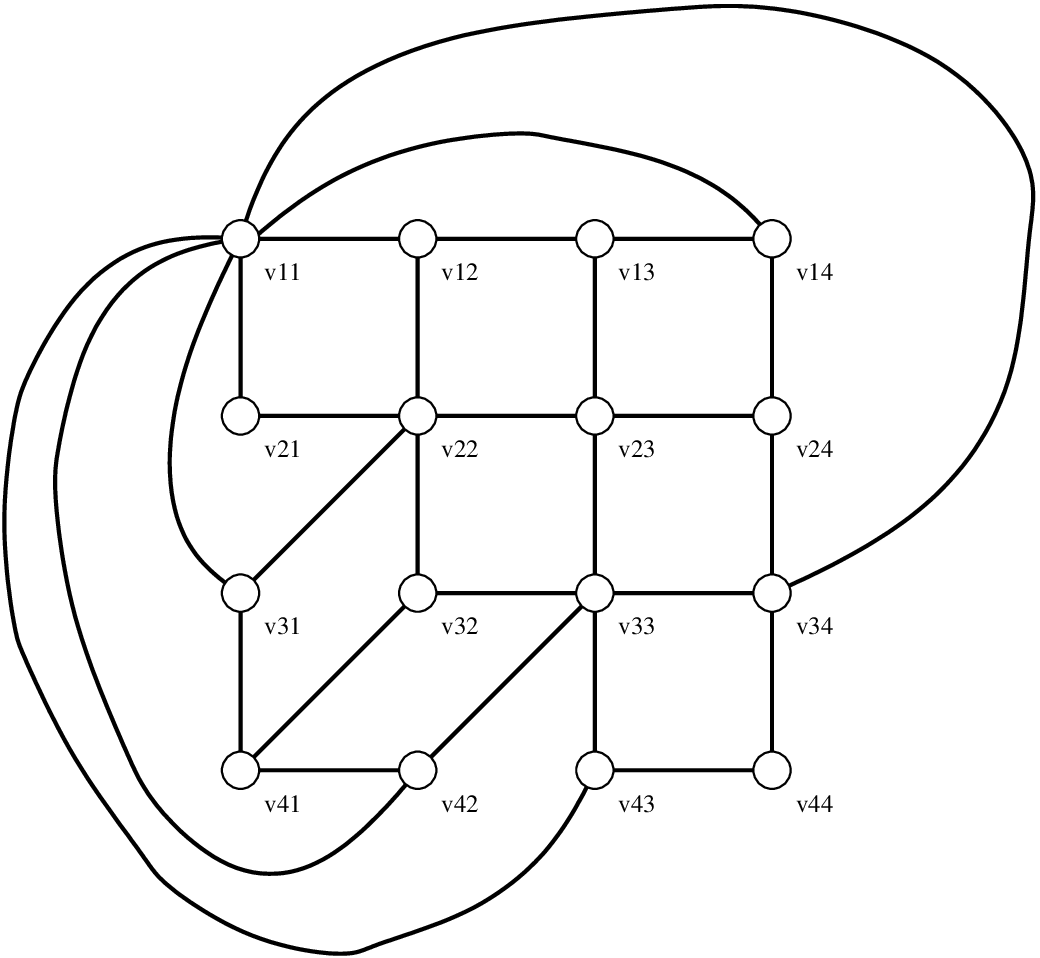}
  \\
  \includegraphics[height=.25\textheight]{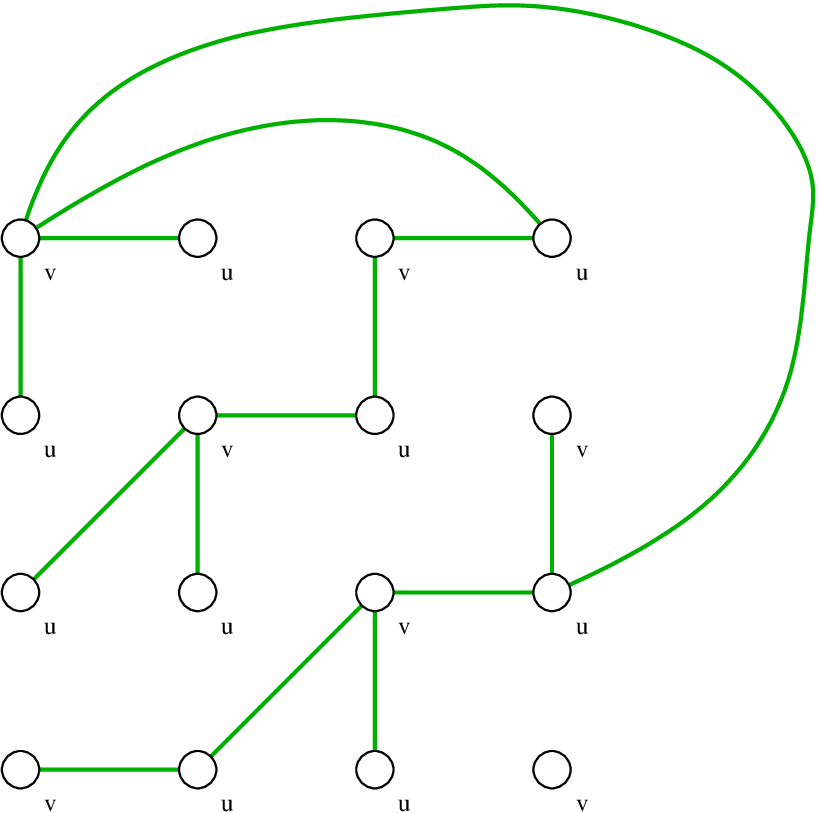}
  \qquad
  \includegraphics[height=.25\textheight]{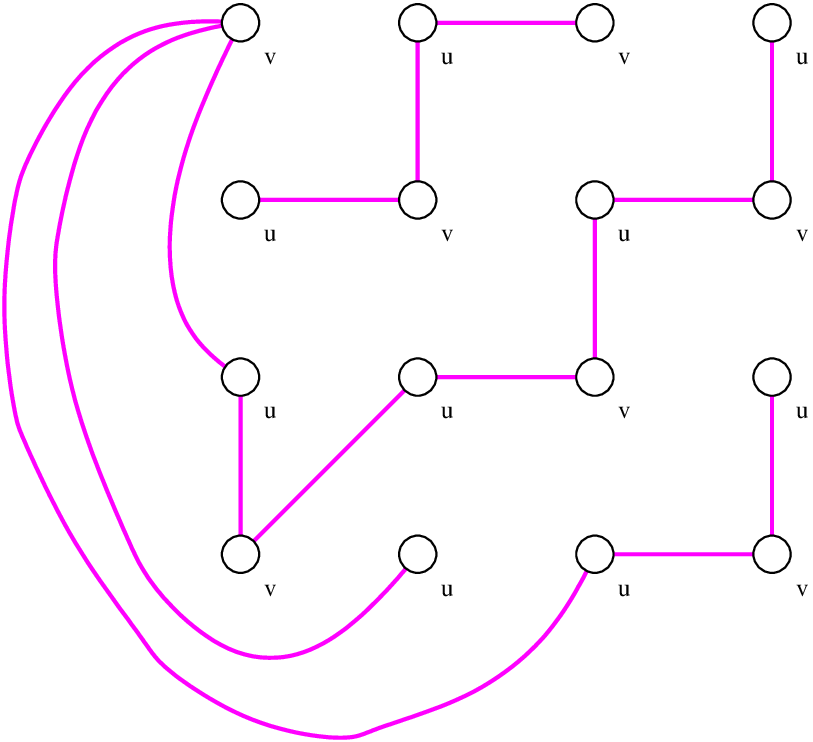}
  \caption{The decomposition of the triangulated $4\times 4$ grid into two forests.}\label{fig7}
\end{figure}

\begin{cor}\label{CorGrid}
  Let $G$ be a $n\times n$ grid for any integer $n \geq 3$ and let
  $Gt$ denote the triangulated grid of $G$. Then every (planar)
  subgraph of $Gt$ satisfies
  $\chi''_{{\cal D}_1,{\cal D}_1}(Gt) \leq 4$.
\end{cor}


\begin{thebibliography}{99}
\bibitem{be69} M. Behzad, The total chromatic number of a graph, a
  survey, Proc. Conference Combinatorial Mathematics (Oxford, England
  1969), Academic Press New York (1970).

\bibitem{bechco67} M. Behzad, G. Chartrand, J.K. Cooper Jr, The 
  numbers of complete graphs, J. London. Math. Society, 42(1967)
  226--228.

\bibitem{bobr16} M. Borowiecki and I. Broere,
Hamiltonicity and Generalised Total Colourings of planar graphs,
Discussiones Mathematicae Graph Theory 36 (2016) 243--257.

\bibitem{bobr97} M. Borowiecki, I. Broere, M. Frick, P. Mih{\'o}k and
  G. Semani{\v s}in, A survey of hereditary properties of graphs,
  Discussiones Mathematicae Graph Theory, 17 (1997) 5--50.

\bibitem{bokemi2011} M. Borowiecki, A. Kemnitz and P. Mih{\'o}k,
  Generalized total Colourings of graphs, Discussiones Mathematicae
  Graph Theory, 31 (2011) 209--222.

\bibitem{ChartLes} G. Chartrand, L. Lesniak and P. Zhang, Graphs \& Digraphs,
  fifth edition, Chapman \& Hall/CRC, 2010, ISBN 1-43982-627-7.

\bibitem{Karafova} G. Karafov\'a, Generalized Fractional Total
  Colorings of Complete Graphs, Discussiones Mathematicae Graph
  Theory, 33 (4) (2013) 665--676.

\bibitem{Kruskal} J.B. Kruskal, On the shortest spanning tree of a
  graph and the traveling salesman problem, Proc. Amer. Math. Soc., 7
  (1956) 48--50.

\bibitem{Vija71} N. Vijayaditya, On total chromatic number of a graph,
  Journal of the London Mathematical Society (2), 3 (1971), 405--408.
\end{thebibliography}
\end{document}